\documentclass[a4paper,14pt]{amsart}

\newtheorem{theo}{Theorem}[section]
\newtheorem{ex}[theo]{Example}
\newtheorem{prop}[theo]{Proposition}
\newtheorem{lem}[theo]{Lemma}
\newtheorem{cor}[theo]{Corollary}

\newtheorem{question}[theo]{Question}

\newtheorem{rema}[theo]{Remark}

\def \kbar {{\bar k}}

\def \dem {\paragraph{ \em Proof. }}

\def \Romannumeral #1 {\expandafter\uppercase\expandafter {\romannumeral #1} }
\def \br {{\rm{Br \,}}}

\def \pic {{\rm {Pic\,}}}
\def \pico {{\rm {Pic^0\,}}}
\def \Div {{\rm{Div\,}}}
\def \gal {{\rm{Gal\,}}}

\def \calo {{\mathcal O}}
\def \T {{\mathcal T}}
\def \M {{\mathcal M}}
\def \N {{\mathcal N}}

\def \cala {{\mathcal A}}
\def \calf {{\mathcal F}}
\def \calg {{\mathcal G}}

\def \spec {{\rm{Spec\,}}}

\def\ov{\overline}

\def \Z {{\bf Z}}
\def \Q {{\bf Q}}

\def \ker {{\rm {Ker}}}
\def \G {{\bf G}_m}

\def\smallsquare{\vbox{\hrule\hbox{\vrule height 1 ex\kern 1 ex\vrule}\hrule}}
\def\enddem{\hfill \smallsquare\vskip 3mm}

\def \abstract{\paragraph{Abstract. }}
\usepackage{amscd}
\usepackage{amssymb}
\usepackage{amsmath}

\everymath{\displaystyle}

\title{On Tate--Shafarevich groups of one-dimensional families of commutative group schemes  over number fields}
\author{David Harari and Tam\'as Szamuely}
\address{Laboratoire de Math\'ematiques d'Orsay, Univ. Paris-Sud, CNRS,
Universit\'e Paris-Saclay, 91405 Orsay, France}
\email{David.Harari@math.u-psud.fr}
\address{Dipartimento di Matematica, Universit\`a di Pisa, Largo Bruno Pontecorvo 5, 56127 Pisa, Italy}
\email{tamas.szamuely@unipi.it}

\date{\today}

\DeclareFontFamily{U}{wncy}{}
\DeclareFontShape{U}{wncy}{m}{n}{%
   <5>wncyr5%
   <6>wncyr6%
   <7>wncyr7%
   <8>wncyr8%
   <9>wncyr9%
   <10>wncyr10%
   <11>wncyr10%
   <12>wncyr6%
   <14>wncyr7%
   <17>wncyr8%
   <20>wncyr10%
   <25>wncyr10}{}
\DeclareMathAlphabet{\cyrille}{U}{wncy}{m}{n}
\def\Sha{\cyrille X}
\def \R{{\bf R}}

\usepackage{graphicx}

\makeatletter
\DeclareRobustCommand\widecheck[1]{{\mathpalette\@widecheck{#1}}}
\def\@widecheck#1#2{%
    \setbox\z@\hbox{\m@th$#1#2$}%
    \setbox\tw@\hbox{\m@th$#1%
       \widehat{%
          \vrule\@width\z@\@height\ht\z@
          \vrule\@height\z@\@width\wd\z@}$}%
    \dp\tw@-\ht\z@
    \@tempdima\ht\z@ \advance\@tempdima2\ht\tw@ \divide\@tempdima\thr@@
    \setbox\tw@\hbox{%
       \raise\@tempdima\hbox{\scalebox{1}[-1]{\lower\@tempdima\box
\tw@}}}%
    {\ooalign{\box\tw@ \cr \box\z@}}}
\makeatother

\begin{document}
\maketitle

\noindent{\small{\sc Abstract.} Given a smooth geometrically connected curve $C$ over a field $k$ and a smooth commutative group scheme $G$ of finite type over the function field $K$ of $C$ we study the Tate--Shafarevich groups $\Sha^1_C(G)$ given by elements of $H^1(K,G)$ locally trivial at completions of $K$ associated with closed points of $C$. When $G$ comes from a $k$-group scheme and $k$ is a number field (or $k$ is a finitely generated field and $C$ has a $k$-point) we prove finiteness of $\Sha^1_C(G)$ generalizing a result of Sa\"idi and Tamagawa for abelian varieties. We also give examples of nontrivial $\Sha^1_C(G)$ in the case when $G$ is a torus and prove other related statements.}

\section{Introduction}

Consider a smooth, proper, geometrically connected curve $X$ over a field $k$ and denote by $K$ its function field. Given a nonempty open subscheme $C\subset X$, denote by $C^{(1)}$ the set of closed points of $C$ and by $K_c$ the completion of $K$ with respect to the discrete valuation associated with a point $c\in C^{(1)}$.

Now assume $G$ is a commutative group scheme over $K$. We define the Tate--Shafarevich groups of $G$ relative to $C$ by
$$\Sha^i_C(G):=\ker (H^i(K,G) \to \prod_{c \in C^{(1)}} H^i(K_c,G))$$
for each $i\geq 0$.

The case where $i=1$ and $G=A$ is an abelian variety has been thoroughly studied in the beautiful paper \cite{saiditag} by Sa\"\i di and Tamagawa. One of their main unconditional results is the finiteness of  $\Sha^1_C(A)$ for an isotrivial abelian variety $A$ and a base field $k$ finitely generated over $\Q$.

Our main result extends (a slightly weaker form of) their theorem to arbitrary commutative group schemes of finite type.

\begin{theo}\label{main} Let $G_k$ be a commutative group scheme of finite type over $k$, and set $G:=G_k \times_k K$. The group $\Sha^1_C(G)$ is finite in each of the following cases:
\begin{enumerate}
  \item[(a)] $k$ is a number field;
  \item[(b)] $k$ is finitely generated over $\Q$ and $C$ has a rational point.
\end{enumerate}
\end{theo}

The proof will be given in the next section. Of course, our main new contribution is the case of groups of multiplicative type. In this respect we ask:

\begin{question}\label{mainq} For $k$ a finitely generated extension of $\Q$ and $G$ a group of multiplicative type over $K$ is the group $\Sha^1_C(G)$ always finite?
\end{question}

Note that in (\cite{saiditag}, Proposition A) Sa\"\i di and Tamagawa prove that for $G=A$ an abelian variety over $K$ the $N$-torsion subgroup of $\Sha^1_C(A)$ is finite for every $N>0$. The analogous statement for groups of multiplicative type would of course be equivalent to a positive answer to the question above. For the moment we know that the answer is positive for $G$ constant (by the theorem above), finite (Proposition \ref{h1prop} below) or a stably $K$-rational torus (Corollary \ref{propct}, pointed out by J-L. Colliot-Th\'el\`ene).

An even more basic question is whether this group is always 0. In Section 3 we show:

\begin{prop}\label{example}
There exist examples of smooth proper curves $X$ defined over number fields and tori $T$ defined over their function fields with $\Sha^1_X(T)\neq 0$.
\end{prop}

We construct both constant and non-constant examples of such tori in Example \ref{exconcrete} below.\smallskip

The last section contains some complements to the results above. For instance, we show that $\Sha^1_C(G)$ is always 0 for a group scheme associated with a finitely generated Galois module over $K$ and that $\Sha^2_C(F)$ is finite for a finite \'etale $K$-group scheme defined over $k$. We also explore a possible strengthening of Question \ref{mainq}. As the reader will see, we deduce Theorem \ref{main} in case (b) from the following stronger \mbox{statement:} if $k$ and $G$ are as in the theorem and $C\subset X$ is an open subcurve having a rational point $c\in C(k)$, then already for the single point $c$ the kernel of the restriction map $H^1(C,G\times_kC)\to H^1(K_c,G)$ has finite image in $H^1(K,G)$. However, in Example \ref{exinfinite} below we shall give an example of a non-constant torus over $K$ where such a kernel is infinite. This shows the limitations of the methods of the present paper and hints at the potential difficulty of Question \ref{mainq} in the non-constant case.\smallskip

We would like to mention that in the recent preprint \cite{rr} A. and I. Rapinchuk prove finiteness of another kind of Tate--Shafarevich group associated with a torus defined over a finitely generated field. In their context local triviality is imposed with respect to all discrete valuations coming from codimension 1 points of a normal model of the function field that is of finite type over $\Z$. Ultimately their result relies on fundamental finiteness theorems in \'etale cohomology. Though in the present paper we only consider function fields of dimension 1, we impose a weaker local triviality condition and such a reduction does not seem to work.  \smallskip

We are grateful to Jean-Louis Colliot-Th\'el\`ene and Olivier Wittenberg for their pertinent remarks and to the referee for their careful reading of the text.
We dedicate this note to the memory of our dear friends Joss Beaumont and Jo Cavalier with whom we had the privilege of discussing similar matters in the past.

\section{Constant commutative group schemes}

In this section we prove Theorem \ref{main} of which we keep the notation and assumptions: $k$ is a finitely generated field over $\Q$, $C$ a smooth geometrically connected $k$-curve with function field $K$, and  $G_k$ a commutative group scheme of finite type over $k$, with base change $G:=G_k \times_k K$ to $K$.

We start with a couple of lemmas, of which the first two will only serve in case (a) of the theorem.

\begin{lem}\label{lemmt}
Assume $k$ is a number field. There exists a finite set $S$ of places of $k$ containing all archimedean places such that the kernel of the map
\begin{equation}\label{shageo}
H^1(k,G_k) \to \prod_{c \in C^{(1)}}
	H^1(k(c), G_k)
\end{equation}
is contained in the kernel $\Sha^S(G_k)$ of the map
\begin{equation}\label{shaar}
H^1(k,G_k) \to \prod_{v \not \in S} H^1(k_v,G_k).
\end{equation}

\end{lem}

\begin{dem}
We find a finite set $S$ of places of $k$
containing all archimedean places such that the curve $C$
has $k_v ^h$-points for every $v \not \in S$, where $k_v^h$ is the
henselisation of $k$ at $v$. Such an $S$ exists by the Lang-Weil estimates
and Hensel's Lemma. Since
$k_v^h$ is algebraic over $k$, for every $v\not\in S$ the curve $C$ has a point
over a finite extension $k'$ of $k$ such that $k' \subset k_v^h \subset k_v$.
In particular $C$ has a closed point $c$ such that the residue field
$k(c)$ embeds into $k'$, which implies that the restriction of each element of the kernel of (\ref{shageo}) to
$H^1(k', G_k)$ is zero.
\end{dem}

Concerning the group $\Sha^S(G_k)$ we have the following finiteness statement.

\begin{lem}\label{pt} Still assuming $k$ is a number field, suppose $G_k$ is an extension of an abelian variety $A_k$ by a group of multiplicative type $N_k$.
If $S$ is a finite set of places of $k$ containing all archimedean places, the group $$\Sha^S(G_k)=\ker [H^1(k,G_k) \to \prod_{v \not \in S} H^1(k_v,G_k)]$$ has finite $m$-torsion for all $m>0$.
\end{lem}

\begin{proof}
First consider a place $v\in S$. In the exact sequence
$$
H^1(k_v,N_k)\to H^1(k_v,G_k)\to H^1(k_v,A_k)
$$
the group $H^1(k_v,N_k)$ is finite (see \cite{adt}, I.2.4 and I.2.13) and $H^1(k_v,A_k)$ has finite $m$-torsion (see \cite{adt}, I.3.4 and I.3.7). Therefore $H^1(k_v,G_k)$ has finite $m$-torsion. Denoting by $\Omega$ the set of all places of $k$ and setting
\begin{equation}\label{shaclassic}\Sha(G_k)=\ker [H^1(k,G_k) \to \prod_{v\in\Omega} H^1(k_v,G_k)],\end{equation}
the exact sequence of $m$-torsion subgroups
$$0 \to {}_m\Sha(G_k) \to {}_m\Sha^S(G_k) \to \prod_{v \in S}
{}_m H^1(k_v,G_k)$$
shows that it is enough to prove finiteness of ${}_m\Sha(G_k)$.

Fix an open subscheme $U\subset\spec\calo_k$ such that $m$ is invertible on $U$ and moreover  $G_k$ extends to a smooth commutative group scheme $G_U$ over $U$ which is again an extension of an abelian scheme by a group of multiplicative type over $U$. Consider the group
$$
D^1(U,G_U):={\rm Im}(H^1_c(U, G_U)\to H^1(U, G_U))=\ker(H^1(U, G_U)\to\bigoplus_{v\notin U}H^1(k_v, G_k))
$$
where $H^1_c(U, G_U)$ denotes the compact support cohomology group defined, for instance, in (\cite{adt}, \S II.2).
We first contend that the natural map ${}_mD^1(U,G_U)\to {}_mH^1(k, G_k)$ is injective. To see this it will be enough to verify injectivity of the analogous maps $D^1(U,G_U)\{\ell\}\to H^1(k, G_k)\{\ell\}$ on $\ell$-primary torsion subgroups for all prime divisors $\ell$ of $m$. Observe that the connected component $G^\circ_U$ of the identity in $G_U$ is a semiabelian scheme and the component group $F_U=G_U/G^\circ_U$ is a finite \'etale group scheme. Now consider the commutative diagram with exact rows
\begin{equation}\label{devisGU}
\begin{CD}
H^0(U,F_U) @>>> H^1(U, G_U^\circ) @>>> H^1(U, G_U) @>>> H^1(U,F_U) \\
@VV{\cong}V @VVV @VVV @VVV \\
H^0(k,F_k) @>>> H^1(k, G_k^\circ) @>>> H^1(k, G_k) @>>> H^1(k,F_k).
\end{CD}
\end{equation}
Here the first vertical map is an isomorphism by properness of $F_U$ and the last one is injective by properness of $F_U$ because a generically trivial $F_U$-torsor is trivial by the valuative criterion of properness. The group $H^1(U, G_U^\circ)$ is torsion by (\cite{hasza}, Lemma 3.2 (1)), and its $\ell$-primary torsion injects in $H^1(k, G_k^\circ)$ by (\cite{hasza}, Proposition 4.1 (2)). Injectivity of the third vertical map on $\ell$-primary torsion follows by a diagram chase.

This being said, we observe that the subgroup ${}_m\Sha(G_k)\subset H^1(k, G_k)$ is contained in ${}_mD^1(U,G_U)$. Recall the argument from \cite{hasza}: each $m$-torsion element in $\Sha(G_k)$ comes from ${}_mD^1(V,G_V)$ for a suitable open subscheme $V\subset U$, hence from ${}_mD^1(U,G_U)$ by covariant functoriality of the groups $D^1(U,G_U)$ for open immersions (see the proof of Lemma 4.7 in \cite{hasza}).
It thus remains to establish the finiteness of ${}_mD^1(U,G_U)$. In the exact upper row of diagram (\ref{devisGU})
the groups $H^i(U,F_U)$ are finite for $i=0,1$ (for the case $i=1$ see \cite{adt}, remark after Theorem 3.1 $b)$) and the groups ${}_mH^1(U,G^\circ_U)$ are known to be finite for all $m>0$ by (\cite{hasza}, Lemma 3.2 (2)). The required finiteness follows since $D^1(U,G_U)$ is a subgroup of $H^1(U,G_U)$.
\end{proof}

Finally we include for the sake of reference a well-known lemma that is in fact valid over an arbitrary field $k$ (or even for an arbitrary integral Dedekind scheme in place of $C$).

\begin{lem}\label{harder}
For $C'\subset C$ an open subscheme and $\calg$ a smooth commutative group scheme of finite type over $C$ with generic fibre $G$ we have an inclusion of subgroups
$$
{\rm Ker}(H^1(C',\calg)\to\bigoplus_{v\in C\setminus C'}H^1(K_v,G))\subset {\rm Im}(H^1(C, \calg)\to H^1(C',\calg))
$$
in $H^1(C',\calg)$.
\end{lem}

\begin{dem}
Part of the localization sequence in \'etale cohomology reads
$$
H^1(C, \calg)\to H^1(C',\calg)\to\bigoplus_{v\in C\setminus C'}H^2_v(C,\calg),
$$
so we are reduced to establishing an inclusion
$$
{\rm Ker}(H^1(C',\calg)\to\bigoplus_{v\in C\setminus C'}H^1(K_v,G))\subset {\rm Ker}(H^1(C',\calg)\to\bigoplus_{v\in C\setminus C'}H^2_v(C,\calg)).
$$
We have excision isomorphisms
$
H^2_v(C,\calg)\cong H^2_v(\calo_v^h,\calg)
$,
where $\calo_v^h$ is the henselization of the local ring of $C$ at $v$. A second localization sequence for $\calo_v^h$ reads
$$
H^1(\calo_v^h,\calg)\to H^1(K_v^h,G)\to H^2_v(\calo_v^h,\calg)
$$
where $K_v^h$ is the fraction field of $\calo_v^h$. The composition of restriction maps $H^1(C',\calg)\to H^1(K,G)\to H^1(K_v^h,G)$ makes the diagram
$$
\begin{CD}
H^1(C',\calg) @>>> H^2_v(C,\calg)\\
@VVV @VV\cong V\\
H^1(K_v^h,G) @>>> H^2_v(\calo_v^h,\calg)\end{CD}
$$
commute by functoriality of the localization sequence. Thus our claim follows from the fact that the natural map $H^1(K_v^h,G)\to H^1(K_v,G)$ is an isomorphism (by the same argument as e.g. in the proof of \cite{hasza}, Lemma 2.7).
\end{dem}

\begin{proof}[Proof of Theorem \ref{main}] By Chevalley's theorem our $G_k$ is an extension of an abelian variety $A_k$ by an affine group scheme $N_k$.
 The unipotent part $U_k$ of $N_k$ is a normal subgroup scheme of $G_k$ isomorphic to a product of copies of ${\bf G}_a$, whence $H^1(K, U_K)=0$. This shows that the natural map $H^1(K, G)\to H^1(K, G/U_K)$ is injective, so in proving the theorem we may replace $G_k$ by $G_k/U_k$, i.e. assume $N_k$ is of multiplicative type. We set $\calg:=G_k\times_kC$, $\cala:=A_k\times_kC$ and $\N:=N_k\times_kC$.

We first show that the group $\Sha^1_C(G)$ is contained in the image of the group $${\mathcal K}^1(C,\calg):=\ker(H^1(C,\calg) \to \prod_{c \in C^{(1)}} H^1(k(c),\calg_c))$$ in $H^1(K,G)$, where $\calg_c$ denotes the fibre of $\calg$ over the closed point $c\in C$. Each element of $\Sha^1_C(G)$ lifts  to a locally trivial element in $H^1(C', \calg)$ for some open $C'\subset C$, hence to a locally trivial element in $H^1(C, \calg)$ by Lemma \ref{harder}. The restriction maps $H^1(C', \calg)\to H^1(K_c, \calg)$ factor through the maps $H^1(\calo_c, \calg)\to H^1(K_c, \calg)$, where $\calo_c$ is the completion of the local ring of $c$ and $H^1(\calo_c, \calg)\cong H^1(k(c), \calg_c)$ by smoothness of $\calg$. To conclude the proof of the claim we show injectivity of the maps $H^1(\calo_c, \calg)\to H^1(K_c, \calg)$. Consider the commutative diagram with exact rows
$$
\begin{CD}
H^0(\calo_c, \cala) @>>> H^1(\calo_c, \N) @>>> H^1(\calo_c, \calg) @>>> H^1(\calo_c, \cala)\\
@V{\cong}VV @VVV @VVV @VVV\\
H^0(K_c, \cala) @>>> H^1(K_c, \N) @>>> H^1(K_c, \calg) @>>> H^1(K_c, \cala).
\end{CD}
$$
Properness of $\cala$ implies that the first vertical map is an isomorphism and the last one is injective. The second vertical map is injective, for instance, because of the general result of Nisnevich \cite{nis}, whence injectivity of the third map follows by a diagram chase.

Now since the group $H^1(K,G)$ is torsion (\cite{dhbook},
Corollary~4.23),
it is sufficient to prove that the group ${\mathcal K}^1(C,\calg)$ is finitely generated. We first prove this after passing to a suitable finite \'etale cover of $C$.
Choose a finite
Galois field extension $k_1$ of $k$
such that the base change $N_{k_1}$ splits as a finite direct product of copies of $\G$ and of finite constant group schemes isomorphic to $\Z/m\Z$ for some $m\in\Z$.
Set $D:=C \times_k k_1$, and consider the commutative diagram with exact row
$$
\begin{CD}
H^1(D, \N) @>\rho>> H^1(D, \calg) @>>> H^1(D, \cala) \\
&& @VVV @VVV \\
&& \prod_{d \in D^{(1)}} H^1(k_1(d),G_k) @>>> \prod_{d \in D^{(1)}} H^1(k_1(d),A_k).
\end{CD}
$$
We show that the kernel ${\mathcal K}^1(D,\calg)$ of the first vertical map is finitely generated.
For this we begin by observing that the kernel ${\mathcal K}^1(D,\cala)$ of the second vertical map is finite. Indeed, the map $H^1(D,\cala)\to H^1(k_1(D), A_{k_1(D)})$ is injective because a generically trivial $\cala$-torsor is trivial by the valuative criterion of properness, so we may identify  $H^1(D,\cala)$ with a subgroup of $H^1(k_1(D), A_{k_1(D)})$. On the other hand, as in the previous paragraph, an element of $H^1(D,\cala)$ with trivial image in $H^1(k_1(d),A_k)\cong H^1(\calo_d, \cala)$ has trivial image by the restriction map $H^1(k_1(D), A_{k_1(D)})\to H^1(k_1(D)_d, A_{k_1(D)})$. Thus we have an inclusion ${\mathcal K}^1(D,\cala)\subset\Sha^1_D(A_{k_1(D)})$, and the latter group is finite by (\cite{saiditag}, Theorem 4.1).  This being said, it suffices to show that  $\rho(H^1(D,\N))\cap{\mathcal K}^1(D,\calg)$ is finitely generated. The group $H^1(D,\N)$ splits as a finite product of copies of $H^1(D,\G)$ and of groups of the form  $H^1(D,\Z/m\Z)$. The unique smooth $k_1$-compactification $Y$ of $D$ has finitely generated Picard group by the Mordell--Weil theorem, whence $H^1(D,\G)\cong\pic D$ is finitely generated as a quotient of $\pic Y$. As for $H^1(D,\Z/m\Z)$, it sits in an exact sequence
$$
0\to H^1(k_1,\Z/m\Z)\to H^1(D,\Z/m\Z)\to H^1(\overline D,\Z/m\Z)
$$
where $\overline D$ is the base change of $D$ to an algebraic closure. Since the cohomology group $H^1(\overline D,\Z/m\Z)$ is finite (by basic theorems of \'etale cohomology or by finiteness of the mod $m$ quotient of the abelianized fundamental group $\pi_1^{\rm ab}(\overline D)$) we reduce to proving finiteness of the intersection $\rho(H^1(k_1,\Z/m\Z))\cap{\mathcal K}^1(D,\calg)$. But this is an $m$-torsion group contained in the kernel of the map
\begin{equation}\label{map}
H^1(k_1,G_{k_1}) \to \prod_{d \in D^{(1)}}
	H^1(k_1(d), G_{k_1})
\end{equation}
(here we identify $H^1(k_1,G_{k_1})$ with its image in $H^1(D,\calg)$).

In case (a) of the theorem this kernel is contained in the group $\Sha^S(G_{k_1})$ for a suitable $S$ by Lemma \ref{lemmt}. But the latter group has finite $m$-torsion part by Lemma \ref{pt}, and we are done. In case (b) the map (\ref{map}) is injective because the $k$-point of $C$ induces a $k_1$-point of $D$ and the map is the identity on the corresponding component of the product.

We can now prove that ${\mathcal K}^1(C,\calg)$ is finitely generated. With the notation $\Gamma:=\gal(k_1|k)$ we have a commutative diagram with exact row
$$
\begin{CD}
H^1(\Gamma, \calg(D)) @>>> H^1(C, \calg) @>>> H^1(D, \calg) \\
&& @VVV @VVV \\
&& \prod_{c \in C^{(1)}} H^1(k(c),G_k) @>>> \prod_{d \in D^{(1)}} H^1(k_1(d),
G_{k_1}).
\end{CD}
$$
where the second vertical map has finitely generated kernel by the previous paragraph.  It therefore suffices to verify finiteness of the intersection of ${\rm Im}(H^1(\Gamma, \calg(D))\to H^1(C,\calg))$ with ${\mathcal K}^1(C,\calg)$. Consider the exact sequence
\begin{equation}\label{devisG}
H^1(\Gamma, \N(D))\to H^1(\Gamma, \calg(D))\to H^1(\Gamma, \cala(D)).
\end{equation}
Since $\cala(D)$ equals $A(k_1(D))$ by properness of $\cala$ over $D$ and the latter group is finitely generated by the Mordell--Weil/Lang--N\'eron theorem, we conclude that $H^1(\Gamma, \cala(D))$ is finite (see e.g. \cite{dhbook}, Corollary 1.50). In order to treat $H^1(\Gamma, \N(D))$, set $L:=\N(D)/N_k(k_1)$. It is a finitely generated abelian group because
so is $\G(D)/\G(k_1)=k_1[D]^*/k_1^*$
in view of the exact sequence
$$0 \to k_1^* \to k_1[D]^*  \to \Div_{Y\setminus D} \, Y$$ where the last group denotes divisors on $Y$ supported in the (finite) complement of $D$.
Thus in the exact sequence
\begin{equation}\label{devisN}
H^1(\Gamma,N_k(k_1)) \to H^1(\Gamma,\N(D)) \to H^1(\Gamma,L)\end{equation}
the group $H^1(\Gamma,L)$ is finite for the same reason as above. Hence by exact sequences (\ref{devisG}) and (\ref{devisN}) we reduce to showing finiteness of the intersection $${\rm Im}(H^1(\Gamma,N_k(k_1))\to H^1(C,\calg))\cap{\mathcal K}^1(C,\calg).$$
This group in turn is contained in the image of the group
$$
K_G:=\ker(H^1(k,G_{k}) \to \prod_{c \in C^{(1)}}
	H^1(k(c), G_{k}))
$$
in $H^1(C,\calg)$. But the group $K_G$ is trivial in case (b) and is finite in case (a) because it is of finite exponent (by a restriction-corestriction argument applied to some residue field $k(c)$) and has finite $m$-torsion by Lemmas \ref{lemmt} and  \ref{pt}.
 \end{proof}

\section{Non-triviality of Tate--Shafarevich groups}

In this section we construct examples of tori $T$ with non-trivial $\Sha^1_X(T)$ as stated in Proposition \ref{example}. In the whole section $k$ is a number field, otherwise we keep the notation from the introduction.

We start with some statements about the cohomology of $\G$ that we consider interesting in their own right. Denote by $\Omega_k$ the set of
places of $k$ and by $k_v$ the completion of $k$ at $v \in \Omega_k$. Define the locally trivial part of the Brauer group of the curve $X$ by

$$\br_{\rm lt} (X):=\ker(\br X \to \prod_{v\in\Omega_k}\br (X\times_k{k_v})).$$

\begin{prop} \label{keyprop}
For every nonempty open subcurve $C\subset X$ we have an isomorphism
$$\Sha^2_C(\G)\cong\br_{\rm lt} (X).$$
In particular the group $\Sha^2_C(\G)$ is independent of $C$.
\end{prop}

\dem Suppose that $\alpha \in \br X \subset \br K$ has trivial restriction to $\br (X\times_k{k_v})$ for
every place $v$ of $k$. If $c$ is a closed point of $X$ with residue
field $k(c)$ and $w$ is a place of $k(c)$ extending $v$, the commutative diagram of restriction maps
$$
\begin{CD}
\br X @>>> \br (X\times_k{k_v})\\
@VVV @VVV \\
\br k(c) @>>> \br k(c)_w
\end{CD}
$$
shows that the evaluation $\alpha(c) \in \br (k(c))$ has trivial restriction to
$\br (k(c)_w) $.
Therefore $\alpha(c)=0$ by the Hasse principle for Brauer groups, which in turn shows that the image of $\alpha$ in
$H^2(K_c,\G)$ is zero thanks to the injection $H^2(\calo_c,\G)
\hookrightarrow H^2(K_c,\G)$ and the isomorphism $H^2(\calo_c,\G) \cong
H^2(k(c),\G)$, where $\calo_c$ is the valuation ring of $K_c$. This proves that $\alpha \in \Sha^2_X(\G) \subset \Sha^2_C(\G)$.

Conversely, pick $\alpha_0 \in \Sha^2_C(\G)$. We first show that $\alpha_0$ comes from
$\alpha \in \br X$. Shrinking $C$ if necessary, we may assume that $\alpha_0$ comes from
$\alpha \in \br C$, hence the evaluation $\alpha(c) \in \br (k(c))$ is
trivial for every closed point $c \in C$. For a place $v$ of $k$ consider
the henselization $k_v^h$ of $k$ at $v$. As $k_v^h$ is an algebraic
extension of $k$, the image of every $k$-morphism $\spec k_v^h \to C$
is a closed point $c \in C$, which implies by functoriality
that $\alpha(P_v^h)=0$ for every $k_v^h$-point $P_v^h \in C(k_v^h)$
(since by assumption $\alpha(c)=0$). By Greenberg's approximation theorem,
this also proves that $\alpha(P_v)=0$ for every $k_v$-point
$P_v \in C(k_v)$. Finally, from (\cite{dhduke} Th. 2.1.1)
we conclude that $\alpha \in \br X$.

We now prove that the restriction $\alpha_v \in \br (X \times_k k_v)$
of $\alpha$ is zero for all places $v$ of $k$. Let $k'_v$ be a finite
field extension of $k_v$ and take $P'_v \in X(k'_v)$. By the same
argument as above, we have $\alpha(P'_v)=0$ if $P'_v$ lies above a closed point of $C$, hence for every $P'_v \in X(k'_v)$ since the evaluation map for $\alpha$ is locally constant (\cite{ctskobook}, Proposition 10.5.2). Therefore
$\alpha(z_v)=0$ for all zero-cycles $z_v$  on $X \times_k k_v$.
By Lichtenbaum duality \cite{licht} for the curve $X \times_k k_v$ it follows
that $\alpha_v=0$. This is what we wanted to prove.

\begin{rema}\rm Olivier Wittenberg contributes the following alternative argument for the second part of the above proof which does not use the difficult result from \cite{dhduke}. As above, we lift $\alpha_0 \in \Sha^2_C(\G)$ to $\alpha\in\br C$ and prove that $\alpha(z_v)=0$ for all zero-cycles $z_v$  on $C \times_k k_v$. By a version of Lichtenbaum duality for open curves (\cite{schvh}, Theorem 3.5 and Remark 2.10.2) it follows that $\alpha$ maps to 0 in $\br(C \times_k k_v)$ and therefore $\alpha_0$ maps to 0 in $H^2(K\times_kk_v,\G)$ for all $v$. Hence for a closed point $P\in X$ the residue of $\alpha_0$ in $H^1(k(P),\Q/\Z)$ has trivial image in $H^1(k(P)_w,\Q/\Z)$ for every place $w$ of $k(P)$ and is therefore trivial by the Chebotarev density theorem. This shows that $\alpha_0$ lifts to an element in $\br_{\rm lt} (X)$.
\end{rema}

\begin{question}\rm
Can one generalize the isomorphism of the proposition to $k$-group schemes other than $\G$? In particular, is there an analogous formula for $\Sha^2_C(F_k)$ with $F_k$ a finite $k$-group scheme?
\end{question}

\begin{theo} \label{h2theo}
Let $J$ be the Jacobian variety of the curve $X$. If $X$ has a zero-cycle of degree 1, we have isomorphisms
$$
\Sha^2_C(\G) \cong \Sha(J)
$$
for every nonempty open subcurve $C\subset X$.
\end{theo}

Here $\Sha(J)$ denotes the classical Tate-Shafarevich group of $J$, as in (\ref{shaclassic}).\medskip

\dem  We have to construct an isomorphism $\br_{\rm lt} (X)\cong\Sha(J)$  thanks to Proposition~\ref{keyprop}.
Since $X$ is a curve, the base change $\ov X=X \times_k \kbar$ satisfies  $\br \ov X=0$ by Tsen's theorem.
Thus from the Hochschild-Serre spectral sequence and the vanishing of
$H^3(k,\G)$ one deduces an exact sequence
$$0\to\br k \to \br X \to H^1(k,\pic \ov X) \to 0$$
where the first map is injective by the assumption that $X$ has a zero-cycle of degree 1. More precisely, this assumption gives a splitting of the map $\br k \to \br X$, so the sequence is in fact split exact.

Repeating the argument with $k$ replaced by its completions $k_v$ we obtain a commutative diagram with split exact rows
$$
\begin{CD}
0 @>>> \br k @>>> \br X @>>> H^1(k,\pic \ov X) @>>> 0 \cr
&& @VVV @VVV @VVV \cr
0 @>>> \prod_{v \in \Omega_k} \br k_v @>>> \prod_{v \in \Omega_k} \br (X\times_k{k_v})
@>>> \prod_{v \in \Omega_k} H^1(k_v,\pic \ov X) @>>> 0 .
\end{CD}
$$
Since the left vertical map is injective by the Hasse principle for Brauer groups, we obtain an isomorphism
$$
\br_{\rm lt} (X)\cong \ker(H^1(k,\pic \ov X)\to  \prod_{v \in \Omega_k} H^1(k_v,\pic \ov X)).
$$
On the other hand,
the exact sequence of Galois modules
$$0 \to J(\kbar) \to \pic \ov X \to \Z \to 0$$
induces an isomorphism
$H^1(k,J) \stackrel\sim\to H^1(k,\pic \ov X)$ in view of the vanishing of $H^1(k,\Z)$ and the assumption that $X$ has a zero-cycle of degree 1. Similarly, we have $H^1(k_v,J) \stackrel\sim\to H^1(k_v,\pic \ov X)$ for $v\in\Omega$ and the statement follows.
\enddem

\begin{rema}\rm By refining the arguments in the above proof one can show that even in the case where $X$ has no zero-cycle of degree 1 the finiteness of $\Sha(J)$ implies the finiteness of  $\Sha^2_C(\G)$. See the proof of Proposition \ref{cormain} below for some of the ideas involved.
\end{rema}

We now construct the promised examples of tori.

\begin{prop} \label{normic}
Let $L|K$ be a finite cyclic Galois extension with Galois group $\Gamma$, and let $Y$ be the normalization of $X$ in $L$. Denote by $J_X$ (resp. $J_Y$) the Jacobian variety of $X$ (resp. $Y$) and suppose that both $X$ and $Y$ have a zero-cycle of degree 1.

For the normic torus $T:=\R^1_{L/K} \G$ we have an isomorphism $$\Sha^1_X(T)\cong\ker [\Sha(J_X) \to
\Sha(J_Y)].$$
\end{prop}

\dem
By definition of $T$ we have an exact sequence of tori
$$1 \to T \to \R_{L/K} \G \stackrel{N_{L/K}}{\to} \G \to 1.$$
In view of Hilbert's Theorem 90, Shapiro's lemma and the $2$-periodicity of the cohomology of finite
cyclic groups, the sequence yields isomorphisms
$$H^1(K,T)\cong K^*/N_{L/K} L^*=\hat H^0(\Gamma,L^*) \cong H^2(\Gamma,L^*)=\br(L/K),$$
and likewise with $K$ replaced by a completion $K_c$, $c \in X^{(1)}$.
Therefore
$$\Sha^1_X(T) \cong \Sha^2_X(\G) \cap \br(L/K) \cong
\ker \left(\Sha^2_X(\G) \to \Sha^2_Y(\G)\right).$$

The result now follows from
Theorem~\ref{h2theo}.
\enddem

To complete the proof of Proposition \ref{example} we give a simple concrete example.

\begin{ex}\label{exconcrete}\rm The Selmer curve $S$ is defined in ${\bf P}_{\Q}^2$ by the equation
$$3x^3+4y^3+5z^3=0.$$
It is a genus one curve with points everywhere locally and
no rational point. Its Jacobian $E$ is an elliptic curve over
$k=\Q$. The curve $S$ acquires a rational point over a cyclic extension $k'|k$ of degree $6$, so the
 non-zero class $[S] \in \Sha(E)$ has trivial restriction
to $\Sha(E\times_k{k'})$.
Let $K$ be the function field of $E$ and $L:=Kk'$. Then the normalisation
of $E$ in $L$ is just $E\times_k k'$, and by the previous proposition for the torus $T=\R^1_{L/K} \G$ the curve $S$ represents a nontrivial class in the group
$$\Sha^1_X(T)\cong\ker(\Sha^1(E) \to \Sha^1(E \times_k k')).$$

\smallskip

Observe that in this example the torus $T$ is constant over the function field $K$ of
the elliptic curve $E$. To get a non-constant example over ${\bf P}^1$,
one can simply take the Weil restriction ${\bf R}_{K/\Q(t)} T$.
\end{ex}

\section{Remarks and Complements}

In this section we collect a number of observations related to the results of the previous sections. We begin with the following consequence of Theorem \ref{main}.

\begin{prop}\label{cormain} Let $k$ be a number field.
If $F_k$ is a finite $k$-group scheme and $F=F_k\times_kK$, then the group $\Sha^2_C(F)$ is finite.
\end{prop}

This statement is of some interest as one can show using Ono's lemma (\cite{sansuc}, Lemma 1.7) that a positive answer to Question \ref{mainq} would follow from the finiteness of $\Sha^2_C(F)$ for all finite $K$-group schemes.\medskip

\begin{proof}
There is an exact sequence of $k$-group schemes
\begin{equation}\label{quasres}
0\to F_k\to Q_k\to T_k\to 0
\end{equation}
where $Q_k, T_k$ are $k$-tori and moreover $Q_k$ is quasi-trivial (write the character group of $F_k$ as a quotient of a permutation module, then dualize). Setting $T=T_k\times_kK$ and $Q=Q_k\times_kK$ as usual, we have $H^1(K, Q)=H^1(K_c, Q)=0$ by quasi-triviality of $Q_k$. This shows that (\ref{quasres}) induces an exact sequence
$$
0\to \Sha^1_C(T)\to \Sha^2_C(F)\to \Sha^2_C(Q)
$$
where the first group is finite by Theorem \ref{main}. Since $\Sha^2_C(F)$ is of finite exponent by finiteness of $F$, it suffices to show that $\Sha^2_C(Q)$ has finite $m$-torsion for all $m>0$. To prove this, by quasi-triviality of $Q$ we reduce to the case $Q=\G$ using Shapiro's lemma. Next, Proposition \ref{keyprop} applies and yields an isomorphism $\Sha^2_C(\G)\cong\br_{\rm lt}(X).$ Now consider the commutative diagram with exact rows
{\small $$
\begin{CD}
 H^0(k, \pic\ov X) @>>> \br k @>>> \br X @>>> H^1(k,\pic \ov X)  \cr
 && @VVV @VVV @VVV \cr
 && \prod_{v \in \Omega_k} \br k_v @>>> \prod_{v \in \Omega_k} \br (X\times_k{k_v})
@>>> \prod_{v \in \Omega_k} H^1(k_v,\pic \ov X)
\end{CD}
$$}

\noindent coming from the Hochschild--Serre spectral sequence as in the proof of Theorem \ref{h2theo}.
The kernel of the third vertical map has finite $m$-torsion by finiteness of the $m$-Selmer group of the Jacobian of $X$. The maps $\br k_v\to \br (X\times_k{k_v})$ have a splitting for all but finitely many $v$ because $X(k_v)\neq\emptyset$ for all but finitely many $v$ by the Weil estimates and Hensel's lemma. On the other hand, the image of the first map in the upper row is finite as $\br k$ is torsion and $H^0(k, \pic\ov X)$ is finitely generated. (Recall that to see this latter fact it is enough to show finiteness of $H^0(k, \pico\ov X)$ which can be checked after replacing $k$ by a finite extension where $X$ has a point; then $H^0(k, \pico\ov X)$ identifies with $\pico X$ which is finitely generated by the Mordell--Weil theorem.) Therefore an $m$-torsion element  $\alpha\in\br_{\rm lt}(X)$ with trivial image in $H^1(k,\pic \ov X)$ comes from an element $\beta\in\br k$ that is locally trivial up to finitely many $v$; furthermore, $n\beta=0$ for some $n\geq m$ independent of $\alpha$. The subgroup of such $\beta$ is finite by the Hasse principle for Brauer groups and the finiteness of $n$-torsion in the Brauer group of a local field. A diagram chase concludes the proof.
\end{proof}

We record the following easy consequence which was in fact proven in the course of the above proof:

\begin{cor}
The group $\Sha^2_C(\G)$ has finite $m$-torsion for all $m>0$.
\end{cor}

The following application has been pointed out to us by J-L. Colliot-Th\'el\`ene. It gives the only non-constant tori for which we know the answer to Question \ref{mainq} at present.

\begin{cor}\label{propct}
If $T$ is a stably rational $K$-torus, then $\Sha^1_C(T)$ is finite.
\end{cor}

\begin{proof}
Since the torus $T$ is stably $K$-rational, by a result of Voskresenski{\u \i} \cite{vosk1} (see also \cite{requiv}, Proposition 6) there is an exact sequence of $K$-tori
$$
1\to Q_1\to Q_2\to T\to 1
$$
with $Q_1$ and $Q_2$ quasi-trivial. By Shapiro's lemma and Hilbert's Theorem 90 the exact sequence induces an injection $\Sha^1_C(T)\hookrightarrow \Sha^2_C(Q_1)$. Again by Shapiro's lemma we have an isomorphism $\Sha^2_C(Q_1)\cong \oplus_i \Sha^2_{D_i}(\G)$ for a finite family of (possibly ramified) covers $D_i\to C$. But $\Sha^1_C(T)$ is a group of finite exponent $m$ (say) and the groups $\Sha^2_{D_i}(\G)$ have finite $m$-torsion by the corollary above.
\end{proof}

Next we show that, in contrast to the case of tori, the Tate--Shafarevich group associated with a finitely generated Galois module is always trivial. The proof, which works over an arbitrary Hilbertian field, is an easy variant of an argument in \cite{saiditag}.

\begin{prop} \label{h1prop} Assume $k$ is finitely generated over $\Q$ and
let $M$ be a finitely generated $\gal(\overline K|K)$-module considered as an \'etale locally constant $K$-group scheme.  Then $\Sha^1_C(M)=0$.
\end{prop}

\dem Consider $\alpha \in \Sha^1_C(M)$. Shrinking $C$ if necessary we
 may assume that $M$ extends to a smooth group scheme $\M$
over $C$ and $\alpha$ comes from an element in $H^1(C,\M)$. For a closed point $c\in C$ with residue field $k(c)$
the specialization map $H^1(\calo_c,\M) \to H^1(k(c),\M_c)$
to the fibre $\M_c$ is an isomorphism.
Moreover, the
restriction map $H^1(\calo_c,\M) \to H^1(K_c,M)$ is injective. Indeed, denoting by $\calo_c^{\rm sh}$ the strict henselization of $\calo_c$ and by $K_c^{\rm sh}$ its fraction field, we have a commutative diagram
$$
\begin{CD}
0 @>>>H^1(k(c), H^0(\calo_c^{\rm sh}, \M)) @>>> H^1(\calo_c,\M) @>>> H^1(\calo_c^{\rm sh}, \M) \\
&& @VVV @VVV  \\
0 @>>>H^1(k(c), H^0(K_c^{\rm sh}, \M)) @>>> H^1(K_c,\M)
\end{CD}
$$
whose exact rows come from the Hochschild--Serre spectral sequence. Here the first vertical map is an isomorphism because $\M$ becomes constant on $\calo_c^{\rm sh}$, and the group $H^1(\calo_c^{\rm sh}, \M)$ vanishes because moreover $\calo_c^{\rm sh}$ is simply connected. This justifies the claimed injectivity,
so our $\alpha$ comes from an element in the kernel of the  specialisation maps $s_c : H^1(C,\M) \to H^1(k(c),\M_c)$ for $c\in C$.

Next, we find a finite Galois extension with group $\Gamma$ such that $\alpha$ comes from $H^1(\Gamma, M^{\gal(\ov K|L)})$. Let $D$ be the normalization of $C$ in $L$.
Shrinking again $C$ if necessary,
we can assume that the covering $D\to C$ is \'etale and that $\alpha$ comes from the subgroup $H^1(\Gamma,\M(D))\subset H^1(C,\M)$.
For $c \in C$ let
$\Gamma_c\subset \Gamma$ be the decomposition subgroup of the covering $D/C$ at $c$.
The restriction of $s_c$ to $H^1(\Gamma,\M(D))$ is given by the restriction
map $H^1(\Gamma,\M(D)) \to H^1(\Gamma_c,\M(D))$, where the target can be identified with a subgroup of $H^1(k(c),\M_c)$ because $D\to C$ is an \'etale covering.
Now by the Hilbert irreducibility theorem  there are infinitely many closed points $c\in C$ such that  the fibre of $D\to C$
at $c$ consists of a single closed point and $\Gamma_c=\Gamma$. Thus the restriction map
$H^1(\Gamma,\M) \to H^1(\Gamma_c,\M)$ is an isomorphism and $\alpha=0$.
\enddem

We now make the following observation. In the proof of Theorem \ref{main} (b) we have in fact proven the following stronger \mbox{statement:} if $k$ and $G$ are as in the theorem, $C\subset X$ is such that $G$ extends to a smooth group scheme $\calg$ over $C$ and $c\in C$ is a rational point, then the kernel of the restriction map $H^1(C,\calg)\to H^1(K_c,G)$ has finite image in $H^1(K, G)$. We now show by way of an example that this stronger statement may fail already if $G$ is a non-constant torus over a number field. We shall need the easy lemma:

\begin{lem}
Let $V$ be a smooth and geometrically integral variety over a number field
$k$, with function field $K$. For a $k$-torus $T$ the kernel of the
restriction map
$H^1(k,T) \to H^1(K,T)$
is finite.
\end{lem}

\dem By the Lang--Weil estimates and Hensel's lemma we find a finite
set of places $S$ of $k$ such that $V(k_v) \neq \emptyset$ for all
$v \not \in S$. With notation as in Lemma \ref{lemmt}, we show that the kernel of $H^1(k,T) \to H^1(K,T)$ is contained in the group $\Sha^S(T)$, which is known to be finite (see Lemma \ref{pt}). Take $\alpha \in \ker (H^1(k,T) \to H^1(K,T))$. We find a nonempty open subscheme $U\subset V$ such that the restriction of $\alpha$ to
$H^1(U,T)$ is zero. For $v \not \in S$ we have $U(k_v) \neq  \emptyset$
by the implicit function
theorem, which implies that the
restriction of $\alpha$ to $H^1(k_v,T)$ of $\alpha$ is also zero, as
claimed.
\enddem

We now come to the promised example.

\begin{ex}\label{exinfinite}\rm
Let $k$ be a number field and $D\to C$ a finite \'etale Galois covering of smooth and geometrically
integral $k$-curves, inducing a field extension $L\supset K$ on generic fibres. Assume that $C$ contains a rational point $c \in C(k)$ such that
the fibre of $D\to C$ at $c$ consists of a single closed point $d\in D$
with residue field $l \supset k$. Fix a $k$-torus $T_k$ such that the restriction map $H^1(k,T_k) \to H^1(l,T_k)$ has infinite kernel. For instance, one may take $T_k$ to be the norm torus $\R^1_{l/k} \G$ as $H^1(k,T_k)\cong k^*/N_{l/k} (l^*)$ is known to be infinite by (\cite{dhbook}, ex. 15.1) and $H^1(l,T_k)=0$. Finally, consider the $L$-torus $T_L:=T_k \times_k L$,
denote by $T$ the $K$-torus $T:=\R_{L/K}(T_L)$ obtained by Weil restriction and extend it to a $C$-torus $\T$.

We now show that the kernel of the map $H^1(C,\T)\to H^1(K_c,T)$ has infinite image in $H^1(K,T)$ for every $c\in C(k)$ satisfying the above condition.  Denoting by $\Sigma$ the complement of $C$ in $X$, we may identify the image of $H^1(C,\T)$ in $H^1(K,T)$ with the subgroup $H^1_{\Sigma}(K,T)\subset H^1(K,T)$ consisting of elements unramified outside $\Sigma$.
Let $\Sigma_D$ be the set of closed points of $D$ lying over
a point of $\Sigma$. By basic properties of the Weil restriction we have isomorphisms
$H^1_{\Sigma}(K,T) \simeq H^1_{\Sigma_D}(L,T_L)$ and $H^1(K_c,T)\cong H^1(L_d,T_L)$.
Hence as before it will suffice to check that the kernel of the specialization map $H^1_{\Sigma_D}(L,T_L) \to H^1(l,T_k)$ at $d$ is infinite.
Now $H^1_{\Sigma_D}(L,T_L)$ contains the image of the natural map $H^1(k,T_k)\to H^1(L,T_L)$ which has finite kernel by the above lemma applied to the geometrically integral $k$-curve $D$. On the other hand, the composite map $H^1(k,T_k)\to H^1_{\Sigma_D}(L,T_L)\to H^1(l,T_k)$ is the natural restriction and therefore has infinite kernel by our assumption. The claim follows.
\end{ex}

Finally, we show that if instead of a torus we consider a finite \'etale group scheme, there is no such counterexample, even over an arbitrary field of characteristic 0.

\begin{prop} Assume $k$ is a field of characteristic 0, and $\calf$ is a finite \'etale group scheme over $C$ with generic fibre $F$ over $K$. For every closed point $c\in C$ the kernel of the restriction map
$$
H^1(C,\calf)\to H^1(K_c,F)
$$
is finite.
\end{prop}

\dem As in the proof of Proposition \ref{h1prop}, it is equivalent to study the kernel of the specialization map $H^1(C,\calf)\to H^1(k(c),\calf_c)$, where $\calf_c$ is the fibre of $\calf$ at $c$.

 We start with the case $\calf=\mu_n$. The $n$-torsion subgroup of $\pic X$ is finite, being contained in the finite $n$-torsion subgroup of the abelian variety $\pic^0 (X\times_k\kbar)$. Therefore the $n$-torsion subgroup
$_n \pic C\subset \pic C$ is also finite because $\pic C$ is the quotient of
$\pic X$ by a finitely generated group (a quotient of the group of divisors supported in $X\setminus C)$.

Using the commutative diagram with exact row
$$
\begin{CD}
0 @>>> k[C]^*/k[C]^{*^n} @>>> H^1(C,\mu_n) @>>> _n \pic C  @>>> 0\\
&& @VVV @VVV  \\
&& k(c)^*/k(c)^{*^n} @>\cong>> H^1(k(c),\mu_n)
\end{CD}
$$
it is therefore sufficient to show finiteness of the kernel of the left vertical map.
In the exact sequence
$$0 \to k^* \to k[C]^* \stackrel{{\rm Div}_{X\setminus C}}{\longrightarrow} P_{X\setminus C} \to 0$$
the group $P_{X\setminus C}$ of principal divisors supported in $X\setminus C$ is finitely generated, and therefore $k^*/k^{*n}$ is of finite index in  $k[C]^*/k[C]^{*n}$. Thus we may replace the latter with the former. Furthermore, after replacing $k(c)$ by its Galois closure $k_1$ it is enough to show finiteness of the kernel of the map $k^*/k^{*n}\to k_1^*/k_1^{*n}$, which by Kummer theory and the restriction-inflation sequence identifies with the finite kernel $H^1(\gal(k_1/k),\mu_n)$ of the restriction map
$H^1(k,\mu_n)\to H^1(k_1,\mu_n)$.

In the general case choose a geometrically integral,
finite, Galois \'etale covering $D$ of
$C$ such that $\calf_D:=\calf \times_C D$ is isomorphic to a direct
sum of copies of $\mu_{n_i}$ for various positive integers $n_i$. Set $G=\gal(D/C)$, and let $d \in D$ be a closed point of $D$ lying over $c$.
The Hochschild-Serre
spectral sequence gives the exact row in the commutative diagram
$$
\begin{CD}
0 @>>> H^1 (G,\calf(D)) @>>> H^1(C,\calf) @>>> H^1(D,\calf_D)\\
&&&& @VVV @VVV \\
&&&& H^1(k(c),\calf_c) @>>> H^1(D,\calf_d)
\end{CD}
$$
Here $H^1 (G,\calf(D))$ is finite because both $G$ and $\calf(D)$ are finite.
The kernel of the right vertical map is finite by the case $\calf=\mu_n$, and we are done.
\enddem


\begin{thebibliography}{99}

\bibitem{requiv} J-L. Colliot-Th\'el\`ene, J-J. Sansuc,
La R-\'equivalence sur les tores, {\it Ann. Sci. \'Ecole Norm. Sup.} {\bf 10} (1977),
175--229.

\bibitem{ctskobook} J-L. Colliot-Th\'el\`ene, A. N. Skorobogatov, {\em The Brauer--Grothendieck group}, Springer-Verlag, Cham, 2021.

\bibitem{dhduke} D. Harari, M\'ethode des fibrations et obstruction
de Manin, {\em Duke Math. J. } {\bf 75} (1994), 221--260.

\bibitem{dhbook} \bysame, {\em Galois cohomology and class field theory}, Springer-Verlag--EDP Sciences, 2020.

\bibitem{hasza} D. Harari, T. Szamuely, Arithmetic duality theorems for
  1-motives, {\em J. reine angew. Math.} {\bf 578} (2005), 93--128.

\bibitem{licht} S. Lichtenbaum,  Duality theorems for curves over $p$-adic fields, {\em Invent. Math.} {\bf  7} (1969), 120--136.

\bibitem{adt} J. S. Milne: {\it Arithmetic Duality Theorems},
Second edition, BookSurge, LLC, Charleston, SC, 2006.

\bibitem{nis} Ye. A. Nisnevich, Espaces homog\`enes principaux rationnellement triviaux et
arithm\'etique des sch\'emas en groupes r\'eductifs sur les anneaux de Dedekind, {\em C. R. Acad.
Sci. Paris} 299 (1984) 5--8.

\bibitem{rr} A. and I. Rapinchuk, Some finiteness results for algebraic groups and unramified cohomology over higher-dimensional fields, arXiv preprint {\tt arXiv:2002.01520}, 2020.


\bibitem{saiditag} M. Sa\"idi, A. Tamagawa, On the arithmetic of abelian
varieties,
{\em J. reine angew. Math.} 762 (2020), 1--33.

\bibitem{sansuc} J.-J. Sansuc, Groupe de Brauer et arithm\'etique
des groupes alg\'ebriques lin\'eaires sur un corps de nombres, {\em J. reine angew. Math.}
{\bf 327} (1981), 12--80.

\bibitem{schvh} C. Scheiderer, J. van Hamel,
Cohomology of tori over $p$--adic curves, {\em Math. Ann.} {\bf 326} (2003), 155--183.

\bibitem{vosk1} V. E.  Voskresenski{\u \i}, Birational properties of linear algebraic groups, {\em Izv. Akad. Nauk SSSR Ser. Mat.}, {\bf 34} (1970), 3--19.

\end{thebibliography}
\end{document}